\documentclass[10 pt,reqno]{amsart}
 \usepackage{amsmath,amssymb}
\usepackage{latexsym}
\numberwithin{equation}{section}
\newtheorem{theorem}{Theorem}[section]
\newtheorem{corollary}[theorem]{Corollary}
\newtheorem{defn}[theorem]{Definition}
\newtheorem{lem}[theorem]{Lemma}
\newtheorem{proposition}[theorem]{Proposition}
\title[Averaging theorem for nonlinear Str\"odinger equations]{An averaging theorem for  nonlinear Schr\"odinger  equations with small nonlinearities}
\author{HUANG Guan\\C.M.L.S, Ecole polytechnique}
\begin{document}
\maketitle
\begin{abstract}
     Consider   nonlinear  Schr\"odinger equations with small nonlinearities
     \[\frac{d}{dt}u+i(-\triangle u+V(x)u)=\epsilon \mathcal{P}(\triangle u,u,x),\quad x\in \mathbb{T}^d.\eqno{(*)}\]
     Let $\{\zeta_1(x),\zeta_2(x),\dots\}$ be the $L_2$-basis formed by  eigenfunctions of the operator $-\triangle +V(x)$. For any complex function $u(x)$,  write it as \mbox{$u(x)=\sum_{k\geqslant1}v_k\zeta_k(x)$} and set $I_k(u)=\frac{1}{2}|v_k|^2$. Then for any solution $u(t,x)$ of the linear equation $(*)_{\epsilon=0}$ we have $I(u(t,\cdot))=const$. In this work it is  proved that if $(*)$  is well posed on time-intervals $t\lesssim \epsilon^{-1}$ and satisfies there some mild a-priori assumptions, then  for any its solution  $u^{\epsilon}(t,x)$,  the limiting behavior of the curve  $I(u^{\epsilon}(t,\cdot))$ on  time intervals of order $\epsilon^{-1}$, as $\epsilon\to0$, can be uniquely characterized by solutions of a certain well-posed   effective equation.  
\end{abstract}
\bibliographystyle{plain}
\setcounter{section}{-1}
\section{Introduction}
We consider the Schr\"odinger equation
\begin{equation}
\frac{d}{dt}u+i(-\triangle u+V(x)u)=0,\quad x\in\mathbb{T}^d,\label{ls1}
\end{equation}
and its nonlinear perturbation:
\begin{equation}
\frac{d}{dt}u+i(-\triangle u+V(x)u)=\epsilon \mathcal{P}(\triangle u, \nabla u, u,  x),\quad x\in \mathbb{T}^d,\label{pls1}
\end{equation}
where $\mathcal{P}: \mathbb{C}^{d+2}\times \mathbb{T}^d\to \mathbb{C}$ is a  smooth function, $1\leqslant V(x)\in C^n(\mathbb{T}^d)$ is a  potential (we will assume that $n$ is sufficiently large)  and $\epsilon\in (0,1]$ is the perturbation parameter.   For any $p\in \mathbb{R}$ denote by $H^p$ the Sobolev space of complex-valued periodic functions, provided with the  norm $||\cdot||_p$,
\[||u||_p^2=\Big\langle (-\triangle)^p u,\; u\Big\rangle+\langle u,\;u\rangle, \quad\text{if}\quad p\in\mathbb{N},\]
where $\langle \cdot,\cdot\rangle$ is the real scalar product in $L^2(\mathbb{T}^d)$,
\[\langle u,\;v\rangle=Re \int_{\mathbb{T}^d}u\bar{v}dx,\quad u,\;v\in L^2(\mathbb{T}^d).\]

If $p>\frac{d}{2}+2=p_d$, then the mapping $H^p\to H^{p-2}$, $u(x)\mapsto \mathcal{P}(\triangle u,\nabla u, u,x)$ is smooth (see below Lemma \ref{lem-smooth}). For any $T>0$, a curve $u\in C([0,T], H^p)$, $p>p_d$, is called a solution of (\ref{pls1}) in $H^p$ if it is a mild solution of this equation. That is, if the relation obtained by integrating (\ref{pls1}) in $t$ from $0$ to $s$ holds for any $0\leqslant s\leqslant T$.  
We wish to study long-time behaviours of solutions for (\ref{pls1}) and assume:
\smallskip

\noindent{\bf Assumption A} {\it (a-priori estimate). Fix some $T>0$. For any $p> p_d+2$, there exists $n_1(p)>0$ such that  if $n\geqslant n_1(p)$, then  for any $0<\epsilon\leqslant1$,  the perturbed  equation (\ref{pls1}), provided with initial data
\begin{equation}
u(0)=u_0\in H^p,\label{ind}
\end{equation}
has a unique solution $u(t,x)\in H^p$ such that
\[||u||_p\leqslant C(T,p,||u_0||_p),\quad \text{for}\quad t\in[0,T\epsilon^{-1}].\]}
Here  and below the constant $C$ also  depends  on the potential $V(x)$.

Denote the operator \[A_Vu:=-\triangle u+V(x)u.\]
   Let $\{\zeta_k\}_{k\geqslant 1}$ and $\{\lambda_k\}_{k\geqslant 1}$ be its real eigenfunctions and eigenvalues, ordered in such a way that
   \[1\leqslant \lambda_1\leqslant\lambda_2\leqslant\cdots.\]
    We say that a potential $V(x)$ is {\it non-resonant}  if 
    \begin{equation}\sum_{k=1}^{\infty}\lambda_k s_k \neq 0,\label{non-r}
    \end{equation}
     for every finite non-zero integer vector $(s_1,s_2,\cdots)$. 
For any complex-valued function $u(x)\in H^p$, we denote by \begin{equation}\Psi(u):=v=(v_1,v_2,\cdots),\quad v_j\in\mathbb{C},
\label{Psi}
\end{equation}
the  vector of its Fourier coefficients with respect to the basis $\{\zeta_k\}$, i.e. \mbox{$u(x)= \sum_{k=1}^{\infty} v_k\zeta_k$}. In the space of complex sequences $v$, we introduce the norms
\[|v|_p^2=\sum_{k\geqslant 1}|v_k|^2\lambda_k^p,\quad p\in \mathbb{R},
\]
and define $h^p:=\{v: |v|_p<+\infty\}$.
Denote 
\begin{equation}
I_k=\frac{1}{2}|v_k|^2,\quad \varphi_k=\text{Arg} \;v_k,\quad k\geqslant 1.
\end{equation}
Then $(I,\varphi)\in \mathbb{R}^{\infty}\times\mathbb{T}^{\infty}$ are the action-angles for the linear equation  (\ref{ls1}). That is, in these variables equation (\ref{ls1}) takes the integrable form
\begin{equation}
\frac{d}{dt}I_k=0,\quad \frac{d}{dt}\varphi_k=\lambda_k,\quad k\geqslant 1.
\end{equation}
Abusing notation we will write $v=(I,\varphi)$. 
Define $h_I^p$ to be the weighted $l^1$-space
\[h^p_I:=\Big\{I=(I_1,\dots)\in\mathbb{R}^{\infty}:|I|_p^{\sim}<+\infty\Big\},\quad |I|_p^{\sim}=2\sum_{i=1}^{\infty}\lambda_i^p|I_i|,\]
and consider the mapping 
\[\pi_I: h^p\to h_I^p,\;v\mapsto I,\quad I_j(v)=\frac{1}{2}|v_j|^2,\quad j\geqslant 1.\]
It    is continuous and its image  is the positive octant~\mbox{$h^p_{I+}=\{I\in h^p_I: I_j\geqslant0,\forall j\}$.}

We  mainly concern with the long time behavior of the actions $I(u(t))\in \mathbb{R}^{\infty}_+$ of solutions for the perturbed equation (\ref{pls1}) for $t\lesssim\epsilon^{-1}$.  For this purpose, it is convenient to pass to the slow time $\tau=\epsilon t$ and  write equation (\ref{pls1})  in the action-angle coordinates $(I,\varphi)$:
\begin{equation}
\dot{I}_k=F_k(I,\varphi),\quad \dot{\varphi}_k=\epsilon^{-1}\lambda_k+G_k(I,\varphi),\quad k\geqslant 1,\label{ipls}
\end{equation}
where $I\in \mathbb{R}^{\infty}$, $\varphi\in\mathbb{T}^{\infty}$ and $\mathbb{T}^{\infty}:=\{(\theta_i)_{i\in\mathbb{N}}: \theta_i\in\mathbb{T}\}$ is the infinite-dimensional torus endowed with the Tikhonov toppology. The functions $F_k$ and $G_k$, $k\geqslant1$  represent  the perturbation term $\mathcal{P}$, written in the action-angle coordinates. 
In the finite dimensional situation, the \emph{averaging principle}   is well established for perturbed integrable systems. The principle states that  for    equations 
\[
\frac{d}{dt}I=\epsilon f(I,\varphi),\quad \frac{d}{dt}\varphi=W(I)+\epsilon g(I,\varphi),
\]
where  $I\in \mathbb{R}^M$ and  $\varphi\in\mathbb{T}^m$,  on  time intervals of order $\epsilon^{-1}$ the action components $I(t)$ can be well approximated by solutions of the following averaged equation:
\begin{equation}
\frac{d}{dt}J=\epsilon \langle f\rangle(J),\quad \langle f\rangle (J)=\int_{\mathbb{T}^m}f(J,\varphi)d\varphi.
\label{pifav}
\end{equation}
This assertion has been justified under various non-degeneracy assumptions on the frequency vector $W$ and  the initial data $(I(0),\varphi(0))$ (see~\cite{LM88}).  In this paper we want to prove a version of the  averaging principle for the perturbed  Schr\"odinger equation (\ref{pls1}).  We define a corresponding averaged equation for (\ref{ipls})   as in  (\ref{pifav}):
\begin{equation}
\dot{J_k}=\langle F_k\rangle(J),\quad \langle F_k\rangle(J)=\int_{\mathbb{T}^{\infty}}F_k(J,\varphi)d\varphi,\quad k\geqslant 1,\label{aipls}
\end{equation}
where $d\varphi$ is the Haar measure on $\mathbb{T}^{\infty}$.  But now, in difference with the finite-dimensional case,  the well-posedness of equation (\ref{aipls}) is not obvious, since the map $\langle F\rangle(I)=(\langle F_1\rangle(I),\dots)$ is unbounded and the functions $\langle F_k\rangle(I)$, $k\geqslant 1$, may be not  Lipschitz  with respect to $I$ in $h_{I+}^p$. 
In~\cite{Kuk10}, S. Kuksin observed  that the averaged equation (\ref{aipls}) may be lifted to a  regular `effective equation' on the variable $v\in h^p$, which transforms to (\ref{aipls}) under the projection $\pi_I$. 
 To derive an  effective equation, corresponding to our problem,  we first use mapping $\Psi$ to write (\ref{pls1}) as a system of equation on the vector $v(\tau)$:
\begin{equation}
\dot{v}=\epsilon^{-1}d\Psi(u)(-iA_V(u))+P(v).\label{vpls}
\end{equation}
Here $P(v)$ is the perturbation term $\mathcal{P}$, written in $v$-variables. This equation is singular when $\epsilon\to 0$. The effective equation for (\ref{vpls}) is a certain  regular equation
\begin{equation}
\dot{v}=R(v).
\label{evpls}
\end{equation}
To define the effective vector filed $R(v)$, for any $\theta=(\theta_1,\theta_2,\cdots)\in \mathbb{T}^{\infty}$ let us denote by $\Phi_{\theta}$ the linear operator in the space of complex sequences $(v_1,v_2,\cdots)\in h^p$ which multiplies each  component $v_j$ with $e^{i\theta_j}$.  Rotation $\Phi_{\theta}$ acts on vector fields on the $v$-space, and $R(v)$ is the result of action of $\Phi_{\theta}$ on $P(v)$, averaged in $\theta$:
\[R(v)=\int_{\mathbb{T}^{\infty}}\Phi_{-\theta}P(\Phi_{\theta}v)d\theta.\]
The map $R(v)$ is smooth with respect to $v$ in $h^p$.  Again, we understand solutions for equation (\ref{evpls}) in the mild sense.

We now make the second assumption:
\smallskip

\noindent{\bf Assumption B} {\it (local well-posedness of the effective equation). For any $p > p_d+2$, there exists $n_2(p)>0$ such that if $n\geqslant n_2(p)$, then for  any  initial data $v_0\in h^p$, there exists $T(|v_0|_p)>0$  such that the effective equations (\ref{evpls}) has a unique solution $v\in C([0,T(|v_0|_p)],h^p)$.  Here $T: \mathbb{R}_+\to \mathbb{R}_{>0}$ is an upper  semi-continuous function.}
\smallskip

The main result of this paper is the following statement, where $v^{\epsilon}(\tau)$ is the Fourier transform of a solution $u^{\epsilon}(t,x)$ for the problem (\ref{pls1}), (\ref{ind}) (existing by Assumption A), written in the slow time $\tau=\epsilon t$:
\[v^{\epsilon}(\tau)=\Psi\big(u^{\epsilon}(\epsilon^{-1}\tau)\big), \quad \tau\in[0,T].
\]
We also assume Assumption B. 

\begin{theorem} For any $p>p_d+2$, if $n\geqslant\max\{p,n_1(p),n_2(p)\}$, then there exists $I^0(\cdot)\in C([0,T],h^p_I)$ such that for every $q<p$,
\[I(v^{\epsilon}(\cdot))\underset{\epsilon\to0}\longrightarrow I^0(\cdot)\quad \mbox{in}\quad C([0,T],h^q_I).\]
 Moreover $I^0(\tau )$,  $\tau\in[0,T]$,  solves the averaged equation (\ref{aipls}) with initial data $I^0(0)=I(\Psi(u_0))$, and it may be written as $I^0(\tau)=I(v(\tau))$, where $v(\cdot)$ is the unique solution of the effective equation (\ref{evpls}),   equal to  $\Psi(u_0)$ at $\tau=0$.
\end{theorem}

\begin{proposition} The assumptions A and B hold if  (\ref{pls1}) is a complex Ginzburg-Landau equation
\begin{equation}\dot{u}+\epsilon^{-1}i(\triangle u+V(x)u)=\triangle u-\gamma_R f_p(|u|^2)u-i\gamma_If_q(|u|^2)u, \quad x\in\mathbb{T}^d,
\label{cgl1}
\end{equation}
where the constants $\gamma_R$, $\gamma_I$ satisfy 
\begin{equation}\gamma_R,\; \gamma_I>0,
\label{cgl1c-1}
\end{equation}
the functions $f_p(r)$ and $f_q(r)$ are the monomials $|r|^p$ and $|r|^q$, smoothed out near zero, and
\begin{equation}
0\leqslant p,q<\infty\quad \text{if}\quad d=1,2\quad\text{and}\quad 0\leqslant p,q<\min\{\frac{d}{2},\frac{2}{d-2}\}\quad \text{if}\quad d\geqslant 3.
\label{cgl1c-2}
\end{equation}
\end{proposition}
 This work is a continuation of the research started in \cite{hg2013}, where the author proved a similar averaging principle (not for all but for  typical initial data) for a perturbed KdV equation:
 \begin{equation}
 u_t+u_{xxx}-6uu_x=\epsilon f(u)(x), \;x\in\mathbb{T},\;\int_{\mathbb{T}} u(t,x) dx=0,
 \label{kdv}
 \end{equation}
 assuming the perturbation $\epsilon f(u)(\cdot)$ defines a smoothing mapping $u(\cdot)\mapsto f(u)(\cdot)$.  This additional assumption is necessary to guarantee the existence of  an quasi-invariant measure for the perturbed equation (\ref{kdv}),  which plays an essential role in the proof due to the non-linear nature of the unperturbed equation.  Since in the present  paper we deal with perturbations of a linear  equation,  this restriction is not needed.
 
 In \cite{Kuk11}, a result similar to Theorem 0.1 was proved   for weakly nonlinear stochastic CGL equation (\ref{cgl1}). There are many  works on long-time behaviors of solutions for nonlinear Schr\"odinger equations. E.g.  the averaging   principle was justified in \cite{ KK95}  for solutions of Hamiltonian perturbations of (\ref{ls1}), provided that the potential $V(x)$ is non-degenerated and that the initial data $u_0(x)$ is a sum of finitely many Fourier modes.   Several long-time stability   theorems which are applicable to  small amplitude solutions of   nonlinear Schr\"odinger equations were presented  in \cite{Bam99,BD03, PO99,bourgain2000}. The results in these works describe the dynamics over a time scale much longer than the $ \mathcal{O}(\epsilon^{-1})$ that we consider, precisely, over  a time interval of order $\epsilon^{-m}$, with arbitrary $m$ (even of order $\exp{\epsilon^{-\delta}}$ with $\delta>0$ in \cite{Bam99,PO99,bourgain2000}).  These results are obtained under the assumption that the frequencies are completely resonant or highly non-resonant (Diophantine-type), by using the normal form techniques near an equilibrium (this is the reason for which  they only apply to  small amplitude solutions). See  \cite{bou2005} and references therein  for general theory of normal form for PDEs.  In difference with the  mentioned works,  the  research in this paper is based on the classical averaging method for finite dimensional systems,  characterizing  by the existence of slow-fast variables.  It deals with arbitrary  solution of equation (\ref{pls1}) with sufficiently smooth initial data. Also note that the non-resonance assumption (\ref{non-r}) is significantly weaker than those in the mentioned works. 
 \smallskip

\noindent{\bf Plan of the paper}.  In Section 1 we recall some spectral properties of the operator $A_V$. Section 2 is about the action-angle form of the perturbed linear Schr\"odinger equation (\ref{pls1}). In Section 3 we introduce the averaged equation and the corresponding effective equation.  Theorem 0.1 and Proposition 0.2 are proved in Section~4 and Section 5.

\section{Spectral properties of $A_V$}

As in the introduction, $A_V=-\triangle +V(x)$,  $x\in \mathbb{T}^d$, where $1\leqslant V(x)\in C^n(\mathbb{T}^d)$ and $\{\lambda_k\}_{k\geqslant1}$ are the eigenvalues of $A_V$. According to Weyl's law, the $\lambda_k$, $k\geqslant1$,  satisfy the following asympototics
\[
\lambda_k=C_d k^{2/d}+o(k^{2/d}),\quad k\geqslant1,
\]
  Fix an $L^2$-orthogonal basis of eigenfunctions $\{\zeta_k\}_{k\geqslant 1}$ corresponding to the eigenvalues $\{\lambda_k\}_{k\geqslant 1}$, and define the linear mapping $\Psi$ as (\ref{Psi}).   For any $m\in\mathbb{N}$, we have  $\langle A_V^m u,u\rangle =|v|^2_m$,  where $v=\Psi u$.  Noting that  $\langle A_V^m u,u\rangle$ is equivalent to $||u||_m^2$ for $m=1,\dots, n$, since $V(x)$ is $C^n$-smooth,   we have the following:

\begin{lem} For every integer  $p\in[0,n]$  the linear mapping 
 $\Psi: H^p\to h^p$ is   an isomorphism. 
 \end{lem}
 
 We denote \[C^n_{+1}(\mathbb{T}^d):=\{V(x)\geqslant 1: V(x)\in C^n(\mathbb{T}^d)\}.\]
 For any finite $M\in \mathbb{N}$ consider the mapping
 \[\Lambda^M: C_{+1}^n(\mathbb{T}^d)\to\mathbb{R}^M,\quad V(x)\to (\lambda_1,\cdots,\lambda_M),\]
 and define the open domain $E_M\subset C_{+1}^n(\mathbb{T}^d)$,
 \[E_M:=\{V|\lambda_1<\lambda_2<\cdots<\lambda_M\}.\]
The complement of $E_M$ is a real analytic variety in $C^n(\mathbb{T}^d)$ of codimension at least~2, so $E_M$ is connected. The mapping $\Lambda^M$ is analytic in $E_M$ (see \cite{KK95}).

Let $\mu$ be a Gaussian measure with a non-degenerate correlation operator, supported by the space $C^n(\mathbb{T}^d)$ (see \cite{BG10}).  Then $\mu(C^n_{+1}(\mathbb{T}^d))>0$. Fix $s\in \mathbb{Z}^M\setminus\{0\}$. The set
\[Q_s:=\{V\in E_M|\Lambda^M(V)\cdot s=0\},\]
is closed in $E_M$. Since $\Lambda^M(V)\cdot s\not\equiv 0$ on $E_M$ (e.g. see \cite{KK95}), then $\mu(Q_s)=0$ (see chapter 9 in  \cite{BG10} and the note \cite{BM13}). Since this is true for any $M$ and  $s$ as above, then we have:
\begin{proposition} The non-resonant potentials form a  subset of $C_{+1}^n(\mathbb{T}^d)$ of full \mbox{$\mu$-measure}.
\end{proposition}
\section{Equation (\ref{pls1}) in action-angle variables}

For $k=1,2,\dots$, we denote:
\[\Psi_k: H^p\to\mathbb{C},\quad \Psi_k(u)=v_k,\]
(see (\ref{Psi})).
Let $u(t)$ be a solution of equation (\ref{pls1}). Passing to slow time $\tau=\epsilon t$,  we get  for $v_k=\Psi_k(u(\tau))$ equations
\begin{equation}
\dot{v}_k+i\epsilon^{-1}\lambda_kv_k= \Psi_k(\mathcal{P}(\triangle u,\nabla u, u,x)),\quad k\geqslant 1.\label{pvls}
\end{equation}
Since $I_k(v)=\frac{1}{2}|\Psi_k|^2$ is an integral of motion for the Schr\"odinger equation (\ref{ls1}), we have 
\begin{equation}
\dot{ I}_k=(\Psi_k(\mathcal{P}(\triangle u,\nabla u, u,x)),v_k):=F_k(v),\quad k\geqslant 1\label{pils}
\end{equation}
(Here and below $( \cdot,\cdot)$ indicates the real scalar product in $\mathbb{C}$, i.e. $(u,v)=Re\; u\bar{v}$.)

Denote $\varphi_k=\text{Arg}\;v_k$, if $v_k\neq 0$, and $\varphi_k=0$, if $v_k=0$, $k\geqslant 1$. Using equation (\ref{pvls}), we get
\begin{equation}
\dot{\varphi}_k=\epsilon^{-1}\lambda_k+ |v_k|^{-2}( \Psi_k(\mathcal{P}(\triangle u,\nabla u, u, x)),iv_k),\quad\text{if}\quad v_k\neq 0,\quad k\geqslant 1\label{ppls}
\end{equation}
Denoting for brevity, the vector field in equation (\ref{ppls}) by $ \epsilon^{-1}\lambda_k+G_k(v)$, we rewrite the equation for the pair $(I_k,\varphi_k)(k\geqslant 1)$ as
\begin{equation}
\dot{I}_k=F_k(v)= F_k(I,\varphi),\quad
\dot{\varphi}_k=\epsilon^{-1}\lambda_k+G_k(v).\label{pipls}
\end{equation}
(Note that the second equation has a singularity when $I_k=0$.)
We denote
\[F(I,\varphi)=(F_1(I,\varphi),F_2(I,\varphi),\cdots).\]
The following result is well known, see e.g. Section 5.5.3 in \cite{sobo1996}.

\begin{lem} If $f(x):\mathbb{C}^{m}\to \mathbb{C}^{N}$ is $C^{\infty}$, then the mapping 
\[M_f: H^p(\mathbb{T}^d,\mathbb{C}^m)\to H^p(\mathbb{T}^d,\mathbb{C}^N),\quad u\mapsto f(u),\]
is $C^{\infty}$-smooth for $p>d/2$. Moreover, it is  bounded and  Lipschitz, uniformly on bounded subsets of $H^p(\mathbb{T}^d,\mathbb{C}^m)$.
\label{lem-smooth}
\end{lem}
In the lemma below, $P_k$ and $P_k^j$ are some fixed continuous functions.
\begin{lem}For  any $j, k\in \mathbb{N}$, we have for any $p>p_d$
\smallskip

(i)The function $F_k(v)$ is smooth in each space $h^p$.

(ii) For any $\delta>0$, the function $G_k(v)\chi_{\{I_k\geqslant \delta\}}$ is bounded by $\delta^{-1/2}P_k(|v|_p)$.

(iii)For any  $\delta>0$, the function $\frac{\partial F_k}{\partial I_j}(I,\varphi)\chi_{\{I_j\geqslant\delta\}}$ is bounded by $\delta^{-1/2}P_k^j(|v|_p)$.

(iv) The function $\frac{\partial F_k}{\partial \varphi_j}(I,\varphi)$ is bounded by $P_k^j(|v|_p)$ and  for any $m\in\mathbb{N}$ and any $(I_1,\cdots,I_m)\in\mathbb{R}^m_+$, the fucntion $F_k(I_1,\varphi_1,\cdots,I_m,\varphi_m,0,\cdots)$ is smooth on $\mathbb{T}^m$.
\label{lem-lpc}
\end{lem}
\begin{proof} \quad Item (i) and (ii) follow directly  from (\ref{pils}), (\ref{ppls}), Lemmata 1.1 and  2.1. Item (iii) and (iv) follow directly  from item (i) and the chain rule.
\end{proof}

Denote
\begin{equation}\Pi_{I,\varphi}: h^p\to h^p_{I}\times \mathbb{T}^{\infty}, \quad\Pi_{I,\varphi}(v)=(I(v),\varphi(v)).
\label{abuse-n}
\end{equation}

\begin{defn} Let assumption A holds. Then for any $p\geqslant p_d+2$ and $T>0$, we call a curve $(I(\tau),\varphi(\tau))$, $\tau\in[0,T]$, a regular solution of equation (\ref{pipls}), if there is a solution $u(t)\in H^p$ of equation (\ref{pls1}) such that
\[ \Pi_{I,\varphi}(\Psi(u(\epsilon^{-1}\tau)))=(I(\tau),\varphi(\tau))\in h^p_{I}\times \mathbb{T}^{\infty}, \quad \tau\in[0,T].\]
\end{defn}
Note that if $(I(\tau),\varphi(\tau))$ is a regular solution, then each $I_j(\tau)$ is a $C^1$-function, while $\varphi_j(\tau)$ may be discontinuous at points $\tau$, where $I_j(\tau)=0$.

For any $p\geqslant p_d+2$, let $(I(\tau),\varphi(\tau))$ be a regular solution of (\ref{pipls}) such that $|I(0)|_p\leqslant M_0$. Then by assumption A,  for any $\epsilon>0$ and $T>0$, we have 
\begin{equation}|I(\tau)|^{\sim}_p=\frac{1}{2}|v(p)|_p^2\leqslant C(p,M_0,T),\quad t\in [0,T].\label{lemma2.3}
\end{equation}

\section{ Averaged equation and Effective equation }

     For a function $f$ on a Hilbert space $H$, we write $f\in Lip_{loc}(H)$ if
      \begin{equation} |f(u_1)-f(u_2)|\leqslant P(R)||u_1-u_2||, \quad  \text{if}  \quad ||u_1||, ||u_2||\leqslant R, \label{lip1}
      \end{equation}
    for a suitable continuous function $P$ which depends on $f$.
      Clearly, the set of functions $Lip_{loc}(H)$ is an algebra.  By Lemma \ref{lem-smooth},  
     \begin{equation}
       F_k(v)\in Lip_{loc}(h^p), \quad   k\in \mathbb{N},\; p>p_d.
      \label{lp_fk}
      \end{equation}
    Let $f\in Lip_{loc}(h^{p})$ and $v\in h^{p_1}$, where $ p_1>p$. Denoting by $\Pi^M, M\geqslant 1$ the projection
     \[ \Pi^M: h^0 \mapsto h^0, \quad (v_1,v_2,\cdots) \mapsto (v_1,\cdots,v_M, 0,\cdots),\]
     we have\[|v-\Pi^M v|_{p}\leqslant \lambda_M^{-(p_1-p)/2}|v|_{p_1}.\]
     Accordingly, 
     \begin{equation}
     |f(v)-f(\Pi^M v)|\leqslant P(|v|_{p})\lambda_M^{-(p_1-p)/2}|v|_{p_1}.\label{lip}
     \end{equation}
     We will denote $v^M=(v_1,\dots,v_M)$ and identify $v^M$ with $(v_1,\dots,v_M,0,\dots)$ if needed. Similar notations will be used for vectors $\theta=(\theta_1,\theta_2,\dots)\in\mathbb{T}^{\infty}$ and vectors \mbox{$I=(I_1,\dots)\in h^p_I$}. 
     
     The torus $\mathbb{T}^M$ acts on the space $\Pi_M h^0$ by linear transformations $\Phi_{\theta^M}$, $\theta^M \in \mathbb{T}^M$, where $\Phi_{\theta^M}: (I^M, \varphi^M)\mapsto (I^M, \varphi^M+\theta^M)$. Similarly, the tous $\mathbb{T}^{\infty}$ acts on $h^0$ by linear transformations $\Phi_{\theta}: (I, \varphi) \mapsto(I,\varphi +\theta)$ with $\theta\in \mathbb{T}^{\infty}$ .
     
     For a function $f\in Lip_{loc}(h^{p})$ and any positive integer $N$, we define the average of $f$ in the first $N$ angles as 
     \[\langle f\rangle_N(v)=\int_{\mathbb{T}^N} f\Big((\Phi_{\theta^N}\oplus \text{id})(v)\Big)d\theta^N,\]
     and define the averaging in all angles as \[\langle f\rangle_{\varphi}(v)=\int_{\mathbb{T}^{\infty}}f(\Phi_{\theta}(v))d\theta,\]
     where  $d\theta$ is the Haar measure on $\mathbb{T}^{\infty}$.  We will denote $\langle \cdot\rangle_{\varphi}$ as $\langle\cdot\rangle$ when there is no confusion. The estimate (\ref{lip}) readily implies that \[|\langle f\rangle_N(v)-\langle f\rangle(v)|\leqslant P(R) 	\lambda_N^{-(p_1-p)/2},\quad \text{if} \quad |v|_{p_1}\leqslant R.\]
     Let $v=(I, \varphi)$, then $\langle f\rangle_N $ is a function independent of $\varphi_1, \cdots, \varphi_N$, and $\langle f\rangle$ is independent of $\varphi$. Thus $\langle f\rangle$ can be written as $\langle f\rangle(I)$.
     \smallskip
     
     \begin{lem}(See \cite{KP08}). Let $f\in Lip_{loc}(h^{p})$, then
     \begin{enumerate}
     \item[(i)] Functions $\langle f\rangle_N(v)$ and $\langle f\rangle$ satisfy (\ref{lip1}) with the same function $P$ as $f$ and take the same value at the origin.
     \item[(ii)] They are smooth if $f$ is. If $f$ is $C^{\infty}$-smooth, then for any $M$,  $ \langle f\rangle(I) $ is a smooth function of the first $M$ components  $I_1,\cdots, I_M$ of the vector $I$.
            \end{enumerate}
     \end{lem}
     \begin{proof}
     \quad Item (i) and the first statement of item (ii) is obvious. Notice that $\langle f\rangle(v)=\langle f\rangle(\sqrt{I_1},\dots)$  is even on each variable $\sqrt{I_j}$, 
     $j\geqslant 1$, i.e. \[\langle f\rangle (\dots, -\sqrt{I_j},\dots)=\langle f\rangle(\dots,\sqrt{I_j},\dots),\quad  j\geqslant1.\]
     Now the second statement of item (ii) follows from  Whitney's  theorem (see Lemma~A in the Appendix).
     \end{proof}

  \smallskip
  
   Denote $C^{0+1}(\mathbb{T}^n)$ the set of all Lipschitz functions on $\mathbb{T}^n$. The following result is a version of the classical Weyl theorem.
      
   \begin{lem} Let $f\in C^{0+1}(\mathbb{T}^n)$ for some $n\in\mathbb{N}$.  For any non-resonant vector $\omega\in\mathbb{R}^{n}$  (see (\ref{non-r})) and any $\delta>0$,  there exists $T_0>0$ such that if $T\geqslant T_0$, $g\in C(\mathbb{T}^n)$ and $|g-f|\leqslant \delta/3$, then we have
   \[\Big|\frac{1}{T}\int_0^T g(x_0+\omega t)dt-\langle g\rangle\Big|\leqslant \delta,\]
   uniformly in $x_0\in\mathbb{T}^n$. 
   \end{lem}
   \begin{proof} It is well known that
   for any $\delta>0$ and non-resonant vector $\omega\in\mathbb{R}^n$, there exists $T_0>0$ such that 
   \[\Big|\frac{1}{T}\int_0^T f(x_0+\omega t)dt-\langle f\rangle\Big|\leqslant \delta/3,\quad \forall T\geqslant T_0,\]
    (see e.g. Lemma 2.2 in \cite{hg2013}).
   Therefore if $T\geqslant T_0$, $g\in C(\mathbb{T}^n)$ and $|g-f|\leqslant \delta/3$, then
   \[\begin{split}\Big|\frac{1}{T}\int_0^T g(x_0+\omega t)dt-\langle g\rangle\Big|
   &\leqslant \Big|\frac{1}{T}\int_0^T f(x_0+\omega t)dt-\langle f\rangle\Big|\\
   &+\frac{1}{T}\int_0^T| f(x_0+\omega t)-g(x_0+\omega t)|dt+|\langle f\rangle-\langle g\rangle|\leqslant \delta.
   \end{split}\]
   This finishes the proof of the lemma.
   \end{proof}
   

We  denote $P_k(v)=\Psi_k(\mathcal{P}(\triangle u, \nabla u,u,x))|_{u=\Psi^{-1}v}$, then equations (\ref{pipls}) becomes
\begin{equation}
\dot{I}_k=(v_k,P_k(v)),\quad \dot{\varphi}_k=\epsilon^{-1}\lambda_k+G_k(v),\quad k\geqslant 1.\label{slowi}
\end{equation}
The averaged equations have the form
\begin{equation}
\dot{J}_k=\langle(v_k,P_k)\rangle_{\varphi}(J),\quad k\geqslant 1,\label{averaged1}
\end{equation}
 i.e.
\begin{equation}
\langle(v_k,P_k)\rangle_{\varphi}=\int_{\mathbb{T}^{\infty}}(v_k e^{i\theta_k},P_k(\Phi_{\theta}v))d\theta=(v_k,R_k(v)),\label{averaged2}
\end{equation}
with
\begin{equation}
R_k(v)=\int_{\mathbb{T}}\Phi_{-\theta_k}P_k(\Phi_{\theta})d\theta.
\label{effective-r}
\end{equation}
Similar to equation (\ref{pls1}), for any $T>0$, we call a curve $J\in C([0,T], h^p_I)$ a solution of equation (\ref{averaged1}) if for every $s\in [0,T]$ it satisfies the relation, obtained by integrating~(\ref{averaged1}).

Consider the differential equations 
\begin{equation}
\dot{v}_k=R_k(v),\quad k\geqslant 1.\label{effective1}
\end{equation}
Solutions of this system are defined similar to that of (\ref{pls1}) and (\ref{averaged1}).
Relation (\ref{averaged2}) implies:

\begin{lem}If $v(\cdot)$ satisfies (\ref{effective1}), then $I(v)$ satisfies (\ref{averaged1}).
\end{lem}
 
Following \cite{Kuk10}, we call equations (\ref{effective1}) the {\it effective equation} for the perturbed equation (\ref{pls1}).

\begin{proposition}The effective equation is invariant under the rotation $\Phi_{\theta}$. That is,  if $v(\tau) $ is a solution of (\ref{effective1}), then for each $\theta\in \mathbb{T}^{\infty}$,  $\Phi_{\theta}v(\tau)$  also is a solution.
        \end{proposition}
        
        \begin{proof} \quad Applying $\Phi_{\theta}$ to (\ref{effective1}) we get that
        \[\frac{d}{d\tau}\Phi_{\theta}v=\Phi_{\theta}R(v).\]
      Relation (\ref{effective-r}) implies that operations $R$ and $\Phi_{\theta}$ commute. Therefore
        \[\frac{d}{d\tau}\Phi_{\theta}v=R(\Phi_{\theta}v).\]
        The assertion follows.   
        \end{proof}

\section{Proof of the Averaging theorem}

 In this section we prove the Theorem 0.1 by studying  the behavior of regular solutions of equation (\ref{pipls}).  We fix $p\geqslant p_d+2$,  assume $n\geqslant\max\{p,n_1(p),n_2(p)\}$  and consider 
        $u_0\in H^{p}$. 
         So \begin{equation}
       \Pi_{I,\varphi}(\Psi(u_0))=(I_0,\varphi_0)\in h_{I+}^p\times \mathbb{T}^{\infty}.
       \end{equation}
      We denote 
      \begin{equation}
      B_p(M)=\{ I\in h_{I+}^p: |I|^{\sim}_p\leqslant M\}.
      \end{equation}
       Without loss of generality, we assume $T=1$. 
                 Fix any  $M_0>0$.
      Let \[(I_0,\varphi_0)\in B_p(M_0)\times \mathbb{T}^{\infty}:=\Gamma_0,\] 
      and   let $(I(\tau),\varphi(\tau))$ be a regular solution of system (\ref{pipls}) with $(I(0),\varphi(0))=(I_0,\varphi_0)$. Then by (\ref{lemma2.3}),  there exists $M_1\geqslant M_0$  such that
       \begin{equation}I(\tau)\in B_p(M_1),\quad\tau\in [0,1].
       \end{equation}
     All constants below depend on $M_1$ (i.e. on $M_0$), and usually this dependence is not indicated.
     From the definition of the perturbation and Lemma 2.1 we know that 
           \begin{equation}
           |\mathcal{F}(I,\varphi)|^{\sim}_{p-2}\leqslant C_{M_1},\quad \forall (I,\varphi)\in B_p(M_1)\times \mathbb{T}^{\infty}.\label{boundf}
           \end{equation}

       Recall that we identify  $I^m=(I_1,\dots, I_m)$ with $(I_1,\dots,I_m,0,\dots)$, etc.
       
       Fix any $n_0\in \mathbb{N}$. By (\ref{lp_fk}),      for every  $\rho>0$, there is $m_0\in \mathbb{N}$ , depending only on $n_0$, $M_1$  and $\rho$, such that if $m\geqslant m_0$, then 
     \begin{equation}|F_k(I,\varphi)-F_k(I^{m},\varphi^{m})|\leqslant \rho,\quad \forall (I,\varphi)\in B_p(M_1)\times\mathbb{T}^{\infty},\label{dn1}
     \end{equation}
     where $k=1,\cdots, n_0$.       
     
     From now on, we always assume  that $(I,\varphi)\in B_p(M_1)\times\mathbb{T}^{\infty}$. 
     
     Since $V(x)$ is  non-resonant, then by Lemma 2.2 and Lemma 3.2, for any $\rho>0$,   there exists $ T_0=T_0(\rho,n_0)>0$, such that
     for all $\varphi\in \mathbb{T}^{\infty}$ and $T\geqslant T_0$, 
     \begin{equation}
     \Big|\frac{1}{T} \int_0^TF_k(I^{m_0},\varphi^{m_0}+\Lambda^{m_0}t)dt-\langle F_k\rangle(I^{m_0})\Big|<\rho,\label{dn}
     \end{equation}
     where $k=1,\dots, n_0$.
Due to Lemma \ref{lem-lpc}, we have 
             \begin{equation}
             \begin{split} &|G_j(I,\varphi)|\leqslant \frac{C_0(j,M_1)}{\sqrt{I_j}},\quad \text{if}\quad I_j\neq0,\\
     &|\frac{\partial F_k}{\partial I_j}(I,\varphi)|\leqslant \frac{C_0(k,j, M_1)}{\sqrt{I_j}},\quad \text{if}\quad I_j\neq0,\\
     & |\frac{\partial F_k}{\partial \varphi_j}(I,\varphi)|\leqslant C_0(k,j,M_1).\label{clumsy1}
         \end{split}
         \end{equation}
                      From  Lemma 3.1, we know
                     \begin{equation}
                            |\langle F_k\rangle(I^{m_0})-\langle F_k \rangle(\bar{I}^{m_0})| \leqslant C_1(k,m_0,M_1)|I^{m_0}-\bar{I}^{m_0}|,\label{clumsy2}
                            \end{equation}
                            and by (\ref{lp_fk}),
                            \begin{equation}
                            |F_k(I^{m_0},\varphi^{m_0})-F_k(\bar{I}^{m_0},\bar{\varphi}^{m_0})|\leqslant C_2(k,m_0,M_1)|v^{m_0}-\bar{v}^{m_0}|,\label{clumsy3}
                            \end{equation}
                            where  $\Pi_{I,\varphi}(v^{m_0})=(I^{m_0},\varphi^{m_0})$ (see (\ref{abuse-n})) and  $|\cdot|$ is the $l^{\infty}$-norm.
                            Denote
                            \[C_{M_1}^{n_0,m_0}=m_0\cdot\max\{C_0,C_1,C_2:1\leqslant j\leqslant m_0,1\leqslant k\leqslant n_0\}.\]
                     From now on we shall use the slow time $\tau=\epsilon t$. 
\begin{lem}
For $k=1,\dots,n_0$, the $I_k$-component of any regular solution of (\ref{pipls}) with initial data in $\Gamma_0$  can be written as:\[I_k(\tau)=I_k(0)+\int_0^{\tau}\langle F_k\rangle(I(s))ds +\Xi(\tau),\]
      where for any $\gamma\in(0,1)$ the function $|\Xi(\tau)|$ is bounded  on $[0,1]$ by
\begin{equation}\begin{split}|\Xi(\tau)|                & \leqslant  C_{M_1}^{n_0,m_0}\Big[\frac{T_0\epsilon}{2\gamma^{1/2}}+\frac{ T_0C_{M_1}\epsilon}{2\gamma^{1/2}} +T_0C_{M_1}\epsilon\\
&\quad+4(\gamma+T_0C_{M_1}\epsilon)^{1/2}\Big](\epsilon T_0+ 1)
            +3\rho+ 3\epsilon C_{M_1}T_0 \quad \tau\in[0, 1],
             \end{split}
             \label{main-e}
             \end{equation}
             where $\rho>0$ is arbitary and $T_0=T_0(\rho,n_0)$ is as (\ref{dn}).
             \end{lem}
             \begin{proof}\quad Let us divide the time interval $[0,\tau]$, $\tau\leqslant 1$, into subinterval $[a_i,a_{i+1}]$, $0\leqslant i\leqslant d_0$, such that 
\[a_0=0,a_{d_0}=\tau,\quad a_{d_0}-a_{d_0-1}\leqslant \epsilon T_0,\]
and  $a_{i+1}-a_i=\epsilon T_0$,  for $0\leqslant  i\leqslant d_0-2$. Then $d_0\leqslant (T_0\epsilon)^{-1}+1$.
For  each interval $[a_i,a_{i+1}]$ we define a subset $\Omega(i)\subset\{1,2,\cdots,m_0\}$ in the following way:
      \[l\in \Omega(i) \quad \Longleftrightarrow\quad\exists t\in [a_i,a_{i+1}],\quad I_l(t)<\gamma.\]
      Then if $l\in \Omega(i)$, by (\ref{boundf}) we have \[|I_l(t)|<T_0C_{M_1}\epsilon +\gamma, \quad t\in [a_i,a_{i+1}].\]
     For $I=(I_1,I_2,\cdots)$ and $\varphi=(\varphi_1,\varphi_2,\cdots)$ we set
      \[\kappa_i(I)=\hat{I},\quad\kappa_i(\varphi)=\hat{\varphi},
      \]
      where the vectors $\hat{I}$ and $\hat{\varphi}$ are defined as  follows:
      \[\text{If}\quad l\in \Omega(i),\quad\text{then}\quad \hat{I}_l=0,\hat{\varphi}_l=0,\quad \text{else}\quad \hat{I}_l=I_l,\;\hat{\varphi}_l=\varphi_l.\]
      We abbreviate  $\kappa_i(I,\varphi)=(\kappa_i(I),\kappa_i(\varphi))$.
      
      Below, $k=1,\dots, n_0$. 
      
      Then on $[a_i,a_{i+1}]$, noting $|v^{m_0}-\kappa_i(v^{m_0}|=\sqrt{2}|I^{m_0}-\kappa_i(I^{m_0})|^{1/2}$, and using (\ref{clumsy3}) we have 
      \begin{equation}
      \begin{split}
       &\int_{a_i}^{a_{i+1}}\Big|F_k\Big(I^{m_0}(s),\varphi^{m_0}(s)\Big)-F_k\Big(\kappa_i\big(I^{m_0}(s),\varphi^{m_0}(s)\big)\Big)\Big|ds\\
       &\leqslant\int_{a_i}^{a_{i+1}}C^{n_0,m_0}_{M_1}\sqrt{2}\Big|I^{m_0}(s)-\kappa_i\Big(I^{m_0}(s)\Big)\Big|^{1/2}ds\\
       &\leqslant \epsilon\sqrt{2} T_0C_{M_1}^{n_0,m_0}(\gamma +T_0C_{M_1}\epsilon)^{1/2}. \label{p0}
       \end{split}
      \end{equation}
By (\ref{dn1}), we have 
\begin{equation}\int_0^{\tau}F_k(I(s),\varphi(s))ds=\int_0^{\tau} F_k(I^{m_0}(s),\varphi^{m_0}(s))ds + \xi_1(\tau),\label{xi1}
\end{equation}
      where $ |\xi_1(\tau)|\leqslant \rho \tau$.
      \smallskip

      \noindent{\bf Proposition 1. }\[\int_0^{\tau} F_k\Big(I^{m_0}(s),\varphi^{m_0}(s)\Big)ds =\sum_{i=0}^{d_0}\int_{a_i}^{a_{i+1}}F_k\Big(I^{m_0}(a_i),\varphi^{m_0}(s)\Big)ds +\xi_2(\tau),\]
        where 
      \begin{equation}|\xi_2|\leqslant \frac{1}{2}C^{n_0,m_0}_{M_1}\Big[4\sqrt{2}(\gamma+T_0C_{M_1}\epsilon)^{1/2}+\gamma^{-1/2}T_0C_{M_1}\epsilon\Big](\epsilon T_0+1).\label{p1}
      \end{equation}
      
     \begin{proof} \quad We may write $\xi_2(t)$ as 
    \[\xi_2(t) =\sum_{i=0}^{d_0-1}\int_{a_i}^{a_{i+1}}\Big[F_k\Big(I^{m_0}(s),\varphi^{m_0}(s)\Big)-F_k\Big(I^{m_0}(a_i),\varphi^{m_0}(s)\Big)\Big]ds\\
    :=\sum_{i=0}^{d_0-1} \tilde{I}_i.\]
     For each $i$, by (\ref{boundf}) and (\ref{clumsy1})  we have
     \begin{equation}\begin{split}&\int_{a_i}^{a_{i+1}}|F_k(\kappa_i(I^{m_0}(s)),\varphi^{m_0}(s))-F_k(\kappa_i(I^{m_0}(a_i),\varphi^{m_0}(s))|ds\\
     &\leqslant \int_{a_i}^{a_{i+1}}\gamma^{-1/2}C_{M_1}^{n_0,m_0}|\kappa_i(I^{m_0}(s)-I^{m_0}(a_i))|ds\\
     &\leqslant \frac{1}{2}C_{M_1}^{n_0,m_0}C_{M_1}T_0^2\gamma^{-1/2}\epsilon^2.
     \end{split}\label{p1.1}
     \end{equation}
     Replacing  the integrand $F_k(I^{m_0},\varphi^{m_0})$ by $F_k(\kappa_i(I^{m_0},\varphi^{m_0}))$, using (\ref{p0}) and (\ref{p1.1}), we have 
     \[\tilde{I}_i\leqslant \frac{1}{2}C^{n_0,m_0}_{M_1}[4\sqrt{2}\epsilon T_0(\gamma+T_0C_{M_1}\epsilon)^{1/2}+\gamma^{-1/2}T_0^2C_{M_1}\epsilon^2].
     \]
       The inequality (\ref{p1}) follows.
       \end{proof}
    

    On each subsegment $[a_i,a_{i+1}]$, we now consider  the unperturbed linear dynamics  $\tilde{\varphi}_i(\tau) $ of the angles $\varphi^{m_0}\in \mathbb{T}^{m_0}$ :
    \[\tilde{\varphi}_i(\tau)=\varphi^{m_0}(a_i)+\epsilon^{-1}\Lambda^{m_0}(\tau-a_i)\in\mathbb{T}^{m_0},\quad \tau\in [a_i,a_{i+1}].\]
    \smallskip
    
   \noindent {\bf Proposition 2.}\[\int_0^{\tau} F_k\Big(I^{m_0}(a_i),\varphi^{m_0}(s)\Big)ds=\sum_{i=0}^{d_0-1}\int_{a_i}^{a_{i+1}}F_k\Big(I^{m_0}(a_i),\tilde{\varphi}_i(s)\Big)ds +\xi_3(\tau),\]
    where
     \begin{equation}
     |\xi_3(\tau)|\leqslant [2\sqrt{2}C_{M_1}^{n_0,m_0}(\gamma+T_0C_{M_1}\epsilon)^{1/2}+\frac{T_0\epsilon}{2\gamma}(C^{n_0,m_0}_{M_1})^2](1+\epsilon T_0).\label{p2}
    \end{equation}
    
 \begin{proof}\quad  On each $[a_i,a_{i+1}]$, 
 notice that 
 \[\begin{split}       &\int_{a_i}^{a_{i+1}}\Big|\kappa_i\Big(\varphi^{m_0}(s)-\tilde{\varphi}_i(s)\Big)\Big|ds
 \leqslant \int_{a_i}^{a_{i+1}}\int_{a_i}^{s}\Big| \kappa_i\Big(\epsilon G^{m_0}(I(s^{\prime}),\varphi(s^{\prime}))\Big)\Big|ds^{\prime}ds\\
    &\leqslant \int_{a_i}^{a_{i+1}}\int_{a_i}^{s}C_{M_1}^{n_0,m_0}\epsilon \gamma^{-1/2}ds^{\prime} ds\leqslant \frac{T_0^2\epsilon^2}{2\gamma^{1/2}}C_{M_1}^{n_0,m_0} .
    \end{split}\]
    Here the first inequality comes from equation (\ref{pipls}),  and using (\ref{clumsy1}) we can get the second inequality.     
    Therefore, using again (\ref{clumsy1}), we have
 \[\begin{split}&\int_{a_i}^{a_{i+1}}\Big[F_k\Big(\kappa_i\big(I^{m_0}(a_i),\varphi^{m_0}(s)\big)\Big)-F\Big(\kappa_i\big(I^{m_0}(a_i),\tilde{\varphi}_i(s)\big)\Big)\Big]ds\\   
       &\leqslant \int_{a_i}^{a_{i+1}}C_{M_1}^{n_0,m_0}\Big|\kappa_i\Big(\varphi^{m_0}(s)-\tilde{\varphi}_i(s)\Big)\Big|ds\\
   & \leqslant \frac{T_0^2\epsilon^2}{2\gamma^{1/2}} (C_{M_1}^{n_0,m_0})^2    \end{split}\]
Therefore (\ref{p2}) holds for the same reason as (\ref{p1}).
\end{proof}
      
      We will now compare the integrals $\int_{a_i}^{a_{i+1}}F_k(I^{m_0}(a_i),\tilde{\varphi}_i(s))ds$ with the average values $\langle F_k(I^{m_0}(a_i))\rangle\epsilon T_0$.
      \smallskip

     \noindent{\bf Propositon 3.}\[\sum_{i=0}^{d_0-1}\int_{a_i}^{a_{i+1}}F_k\Big(I^{m_0}(a_i),\tilde{\varphi}_i(s)\Big)ds=\sum_{i=1}^{d_0-1}T_0\langle F_k\rangle\Big(I^{m_0}(a_i)\Big)+\xi_4(\tau),\]
      
      where
      \begin{equation}|\xi_4(\tau)|\leqslant \rho+2C_{M_1}\epsilon T_0.\label{p3}
      \end{equation}
      
      \begin{proof} \quad     For $0\leqslant i\leqslant d_0-2$,  by (\ref{dn})

      \[\Big|\int_{a_i}^{a_{i+1}}\Big[F_k\Big(I^{m_0}(a_i),\tilde{\varphi}_i(s)\Big)-\langle F_k\rangle\Big(I^{m_0}(a_i)\Big)\Big]ds\Big|\leqslant \epsilon\rho T_0.\]
      So\[\sum_{i=0}^{d_0-2}\Big|\int_{a_i}^{a_{i+1}}F_k\Big(I^{m_0}(a_i),\tilde{\varphi}_i(s)\Big)ds-\langle F_k\rangle\Big(I^{m_0}(a_i)\Big)T_0\Big|
      \leqslant (d_0-1)\epsilon\rho T_0. \]
      
          Moreover,
        \[\Big|\int_{a_{d_0-1}}^{\tau}\Big[F_k\Big(I^{m_0}(a_i),\tilde{\varphi}_i(s)\Big)-\langle F_k\rangle\Big(I^{m_0}(a_i)\Big)\Big]ds\Big|\leqslant 2 C_{M_1}\epsilon T_0.\]

      This implies the inequality (\ref{p3}).
      \end{proof}
  

      \noindent{\bf Proposition 4.}
      \[\sum_{i=1}^{d_0-1}(a_{i+1}-a_i)\langle F_k\rangle\Big(I^{m_0}(a_i)\Big)=\int_0^{\tau}\langle F_k\rangle\Big(I^{m_0}(s)\Big)ds +\xi_5(\tau),\]
      where\begin{equation}
      |\xi_5(\tau)|\leqslant\epsilon C_{M_1}C_{M_1}^{n_0,m_0}T_0(\epsilon T_0+1).\label{p4}
      \end{equation}  
         
     \begin{proof}\quad  Indeed, as \[|\xi_5(\tau)|=\Big|\int_0^{\tau}\Big[\langle F_k\rangle\Big(I^{m_0}(s)\Big)ds-\sum_{i=1}^{d_0-1}(a_{i+1}-a_i)\langle F_k\rangle\Big(I^{m_0}(a_i)\Big)\Big|,\]
      
    then  using  (\ref{boundf}) and (\ref{clumsy2}) we get
     \[\begin{split}|\xi_5(t)|&\leqslant \sum_{i=0}^{d_0-1}\int_{s(i,j)}C_{M_1}^{n_0,m_0}|I^{m_0}(s)-I^{m_0}(a_i)|ds \\
     &\leqslant\epsilon^2 \sum_{i=0}^{d_0-1}C_{M_1}C_{M_1}^{n_0,m_0}(T_0)^2\leqslant \epsilon C_{M_1}C_{M_1}^{n_0,m_0}T_0(\epsilon T_0+1).
     \end{split}\]
     \end{proof}
    
    Finally, we have obvious 
   \smallskip 
   
       \noindent{\bf Proposition 5.} \[\int_0^\tau\langle F_k\rangle\Big(I^{m_0}(s)\Big)ds=\int_0^{\tau}\langle F_k\rangle\Big(I(s)\Big)ds +\xi_6(\tau),\]  
      
      and $|\xi_6(\tau)|$ is bounded by $\rho \tau$.

      Gathering the estimates in Propositions 1-5,  we obtain
       \[ I_k(\tau)=I_k(0)+ \int_0^{\tau} F_k\Big(I(s),\varphi(s)\Big)ds
       =I_k(0)+  \int_0^{\tau} \langle F_k\rangle\Big(I(s)\Big)ds +\Xi(\tau),
       \]
      where $\Xi(\tau)|\leqslant  \sum_{i=1}^6|\xi_i(\tau)|$ satisfies (\ref{main-e}).      
      Lemma 4.1 is proved.
      \end{proof}

      \begin{corollary} For any $\bar{\rho}>0$, with  a suitable choice of  $\rho$, $\gamma$ and  $T_0$, the function $|\Xi(t)|$ in Lemma 4.1 can be made less than $\bar{\rho}$, if $\epsilon$ is small enough.
      \end{corollary}
      
      \begin{proof}\quad
      We choose
      $\gamma=\epsilon^{\alpha}, T_0=\epsilon^{-\sigma}, \rho=\frac{\bar{\rho}}{9}$
      with\[1-\alpha/2 -\sigma>0, \;0<\sigma<1.\]Then for $\epsilon$ small enough, we have $|\Xi(t)|<\bar{\rho}$.
      \end{proof}

 For any $(I_0,\varphi_0)\in \Gamma_0$, let  the curve $(I^{\epsilon}(\tau), \varphi^{\epsilon}(\tau))\in h^p_I\times\mathbb{T}^{\infty}$, $\tau\in[0,1]$, be a regular solution of the equation (\ref{slowi}) such that $(I^{\epsilon}(0),\varphi^{\epsilon}(0))=(I_0,\varphi_0)$.

\begin{lem} The family of  curves  $\{I^{\epsilon}(\tau),\tau\in[0,1]\}_{0<\epsilon<1}$ is pre-compact in $C([0,1],h^{p-2}_I)$. Moreover every limiting (as $\epsilon\to0$)  curve  $I^0(\tau)$, $\tau\in[0,1]$ is a solution of  the averaged equation (\ref{averaged1}),  satisfying 
\[|I^0(\tau)|^{\sim}_p\leqslant M_1,\quad \tau\in[0,1].\]
\end{lem}

\begin{proof}\quad Due to (\ref{lemma2.3}) and (\ref{boundf}), we know that  for any $\epsilon\in(0,1)$,
\[|I^{\epsilon}(\tau)|^{\sim}_p\leqslant M_1,\quad|\frac{d}{d\tau}I(\tau)|^{\sim}_{p-2}\leqslant C_{M_1},\quad \tau\in[0,1].\]
Then by the Arzel\`a-Ascoli theorem, we have that the set 
$\mathcal{I}:=\{I^{\epsilon}(\tau),\tau\in[0,1]\}_{0<\epsilon<1} $ is pre-compact in $C([0,1],h^{p-2}_I)$.
Let $\{\rho_m\}_{m\in\mathbb{N}}$ be a sequence such that 
 $\rho_m\searrow 0$.
From Lemma 4.1 and Corollary 4.2, there is $\epsilon_m>0$ such that if $\epsilon\leqslant \epsilon_m$, then for $k=1,\dots,m$, we have
 \begin{equation}
 \begin{split}
&I^{\epsilon}_k(\tau)=I^{\epsilon}_k(0)+\int_0^\tau \langle F_k\rangle(I^{\epsilon}(s))ds+\Xi_k(\tau),\\
& |\Xi_k(\tau)|\leqslant \rho_m,\quad\tau\in[0,1].
\end{split}\label{can1}
\end{equation}
Let $I^0=I^0(\tau)$, $\tau\in[0,1]$ be a limiting curve of the set $\mathcal{I}$ as $\epsilon\to0$. Then we have 
\[I^0\in C([0,1],h^{p-2}_I)\quad\text{and}\quad |I^0(\tau)|^{\sim}_p\leqslant M_1,\quad\tau\in[0,1].\]
By (\ref{can1}),  the curve $I^0(\cdot)$ solves the averaged equation (\ref{averaged1}).
\end{proof}

   For any $\theta\in\mathbb{T}^{\infty}$ and any vector $I\in h_{I+}^p$ we set 
      \[  V_{\theta}(I)=(V_{\theta_1}(I_1),V_{\theta_2}(I_2),\dots),\]
      where
      $\theta=(\theta_1,\theta_2,\dots)$ and  $V_{\theta_j}(I_j)=\sqrt{2I_j}\cos(\theta_j)+i\sqrt{2I_j}\sin(\theta_j)$, for every $ j\geqslant1$. Then $\varphi_j(V_{\theta_j})\equiv\theta_j$, and for each $\theta\in \mathbb{T}^{\infty}$  the map $I\rightarrow V_{\theta}(I)$ is a right inverse of the map $v\rightarrow I(v)$. For any vector $I$ we denote 
      \[I^{>N}=(I_{N+1},I_{N+2},\dots),\quad V^{>N}_{\theta}(I)=(V_{\theta_{N+1}}(I_{N+1}),V_{\theta_{N+2}}(I_{N+2}),\dots).\]

            \begin{lem} (Lifting) Let $I^0(\tau)=(I_k^0(\tau), k\geqslant 1)\in h^p_{I+}$, $\tau\in[0,1]$,  be a solution of the averaged equation (\ref{averaged1}), constructed in Lemma 4.3.  Then, for any $\theta\in \mathbb{T}^{\infty}$, there is a  solution $v(\cdot)$ of the effective equation (\ref{effective1}) such that 
            \begin{equation}
            I(v(\tau))=I^0(\tau),\quad \tau\in [0,1],\quad \text{and} \quad v(0)=V_{\theta}(I^0(0)).\label{3.12}
            \end{equation}
            \end{lem}
           \begin{proof}\footnote{This argument is a simplified version of the proof of Theorem 3.1 in \cite{Kuk10}}\quad For any $m\in\mathbb{N}$, consider the non-autonomous  finite dimensional systems 
            \begin{equation}\dot{I}_k=\langle F_k\rangle\Big(I_1,\cdots, I_m, \Big(I^0(\tau)\Big)^{>m}\Big),\quad k=1,\cdots, m,\label{(3.13)}
            \end{equation}
            \begin{equation}
            \dot{v}_k=R_k\Big(v_1,\dots,v_m,V_{\theta}^{>m}(I^0(\tau))\Big),\quad k=1,\dots,m.\label{(3.14)}
            \end{equation}
            Obviously, $(I^0_1(\tau),\dots,I^0_m(\tau))$, $\tau\in[0,1]$ solves system (\ref{(3.13)}). It is its unique solution  with initial data $(I^0_1(0),\dots, I^0_m(0))$, since by Lemma 3.1  $\langle F_k\rangle$ is smooth with respect to the variables  $(I_1,\dots,I_m)$.
            
            For $\bar{v}_0=(V_{\theta_1}(I^0_1(0)),\dots,V_{\theta_m}(I^0_m(0)))$, system (\ref{(3.14)})  has a unique solution $v^m(\tau)$,  defined for $\tau\in [0,T^{\prime})$, with $v^m(0)=\bar{v}_0$, where $T^{\prime}\leqslant1$ and $v^m(\tau)\xrightarrow{\tau\to T^{\prime}}\infty$ if $T^{\prime}<1$. Due to equality (\ref{averaged2}), $I(v^m)(\tau)$ solves system (\ref{(3.13)}) in time interval $[0,T^{\prime})$. Since $I(v^m(0))=(I^0_1(0),\cdots,I^0_m(0))$, therefore  $T^{\prime}=1$ and 
            \[I(v^m(\tau))\equiv(I^0_1(\tau),\dots, I^0_m(\tau))\quad \mbox{for}\quad 0\leqslant \tau\leqslant 1.\]
            Now denote 
            \[V_m(\tau)=(v^m(\tau),V_{\theta}^{>m}(\tau)),\quad \tau\in[0,1].\]
            For the same reason as in the proof of Lemma 4.3, the family $\{V_m(\tau),\tau\in[0,1]\}_{m\in \mathbb{N}}$ is pre-compact in $C([0,1],h^{p-2})$ and  \[V_m(0)=V_{\theta}(I^0(0)),\quad I(V_m(\tau))=I^0(\tau),\quad\tau\in [0,1],\quad m\in \mathbb{N}.\]
            So any limiting (as $m\rightarrow\infty$) curve  $v(\cdot)$ of the family$\{V_m(\tau),\tau\in[0,1]\}_{m\in \mathbb{N}}$  is a solution of the effective equation (\ref{effective1}), satisfying  equalities (\ref{3.12}).
           The lemma is proved.
            \end{proof}
            
            \begin{lem} (uniqueness) Under the same assumptions of Lemma 4.3, we have   $I^0(\cdot)\in C([0,1],h^p_I)$  and for every $q<p$,         \begin{equation} I^{\epsilon}(\cdot)\underset{\epsilon\to0}\longrightarrow I^0(\cdot)
            \quad\text{in }\quad C([0,1],h_I^q).
            \label{limq}
            \end{equation}
            \end{lem}
           \begin{proof}\quad Let $I^0(\cdot)$ and $J^0(\cdot)$ be two limiting  curves of the family  $\{I^{\epsilon}(\cdot)\}_{0<\epsilon<0}$, as $\epsilon\to0$, in $C([0,1],h_I^{p-2})$. Then by Lemma 4.4, for any $\theta\in \mathbb{T}^{\infty}$, there are solutions $v_I(\cdot)$, $v_J(\cdot)$ of the effective equation (\ref{effective1}) such that for $0\leqslant \tau\leqslant1$,
\[I(v_I(\tau))=I^0(\tau),\quad I(v_J(\tau))=J(\tau),\quad v_I(0)=v_J(0)=v_0=V_{\theta}(I_0).\]
Due to assumption B,  for initial data $v_0$  the effective equation (\ref{effective1}) has a unique solution $v_E(\cdot)\in C\big(\big[0,T(|v_0|_p)\big), h^p\big)$. Therefore 
\begin{equation}v_I(\tau)=v_J(\tau)=v_E(\tau).
\label{t<1}
\end{equation}
This relation holds for $\tau\leqslant1$ if $T(|v_0|_p)>1$ and for $\tau<T(|v_0|_p)$ if $T(|v_0|_p)<1$. But if $T(|v_0|_p)<1$, then $|v_E(\tau)|_p\rightarrow \infty$ as $\tau\to T(|v_0|_p)$. By the construction in Lemmata 4.3 and 4.4, we know $|v_I(\tau)|_p^2\leqslant M_1$ for $\tau\in[0,1]$. Together with (\ref{t<1})  we have that  $T(|v_0|_p)> 1.$
Hence $I^0=J^0$,  $I^0\in C([0,1],h^p_I)$ and 
\begin{equation}
I^{\epsilon}(\cdot)\underset{\epsilon\to0}\rightarrow I^0(\cdot) \quad \text{in}\quad C([0,1],h^{p-2}_I).
\label{limp-2}
\end{equation}

For any $q<p$, assume that the convergence (\ref{limq}) do not holds. Then there exists $\delta>0$ and sequences $\epsilon_{n}, \tau_n\in[0,1]$ such that 
\begin{equation}
\epsilon_n\to0\quad \text{as} \quad n\to\infty\quad\text{and}\quad|I^{\epsilon_n}(\tau_n)-I^0(\tau_n)|^{\sim}_q\geqslant \delta.
\label{lim-not}
\end{equation}
Takes subsequence $\{n_k\}$ such that $\tau_{n_k}\to\tau_0$ as $n_k\to\infty$. Since  the sequence $\{I^{\epsilon_{n_k}}(\tau_{n_k})\}$ is pre-compact in $h^q_I$, and by (\ref{limp-2}), its limiting point as $n_k\to\infty$ equals $I^0(\tau_0)$, so we have $I^{\epsilon_{n_k}}(\tau_{n_k})$ converges to $I^0(\tau_0)$ in $h^q_I$ as $n_k$ goes to $\infty$. This contradicts with (\ref{lim-not}). 
So we  completes the proof of Lemma 4.5 and also  the proof of Theorem~0.1.
\end{proof}
\section{application to  complex Ginzburg-Landau equations}
In this section we prove that assumptions A and B hold for equation (\ref{cgl1}), satisfying (\ref{cgl1c-1}) and (\ref{cgl1c-2}).  

\subsection{Verification of Assumption A} In this subsection, we denote by $|\cdot|_s$ the \mbox{$L^s$-norm}.  Let $u(\tau)$ be a solution of equation (\ref{cgl1}) such that $u(0,x)=u_0$. Then
\[
\begin{split}
\frac{d}{d\tau}||u(\tau)||_0^2&=2\langle u, \dot{u}\rangle=2\langle u,- 
\epsilon^{-1}iA_Vu+\triangle u-\gamma_R|u|^{2p}u-i\gamma_I|u|^{2q}u\rangle,\\
&=-2||u||^2_1+2||u||_0^2-2\gamma_R|u|_{2p+2}^{2p+2}.
\end{split}
\]
  Since $||u||_0^2\leqslant |u|_{2p+2}^2$,  then relation   $||u(\tau_1)||_0>\gamma_R^{-1/2p}=B_2$ implies that  \[\frac{d}{d\tau}||u(\tau_1)||_0^2<0.\]
So  for any $T>0$ we have
\begin{equation}
||u(T)||_0\leqslant \min\{B_2,\;e^T||u_0||_0\}.\label{norm0}
\end{equation}
Now we rewrite equation (\ref{cgl1}) as follows:
\begin{equation}
\dot{u}+\epsilon^{-1}i (\triangle u +V(x)u +\epsilon\gamma_I|u|^{2q}u)=\triangle u-\gamma_R|u|^{2p}u.
\label{cgl-hf}
\end{equation}
For any $k\in\mathbb{N}$, denote
\[||u||_k^{\backsim2}=\langle A^k_V u,u\rangle,\quad A_V=-\triangle +V(x).\]
The l.h.s is a hamiltonian system with the hamiltonian function $\epsilon^{-1}H(u)$,
\[H(u)=\frac{1}{2}\langle A_V u,u\rangle +\frac{\epsilon}{2q+2}|u|_{2q+2}^{2q+2}.\]
We have $d H(u)(v)=\langle A_V u, v\rangle +\epsilon\gamma_I\langle |u|^{2q}u,v\rangle$, and if $v$ is the vector field in the l.h.s of (\ref{cgl-hf}), then $dH(u)(v)=0$. 
So we have
\[
\begin{split}
\frac{d}{d\tau}H(u(\tau))&=-\gamma_R\langle A_Vu,|u|^{2p}u\rangle+\langle A_Vu,\triangle u\rangle\\
&-\epsilon\gamma_I\gamma_R|u|_{2p+2q+2}^{2p+2q+2}+\epsilon \gamma_I\langle |u|^{2q}u,\triangle u\rangle,
\end{split}
\]
Denoting $U_q(x)=\frac{1}{q+1}u^{q+1}$ and $U_p=\frac{1}{p+1}u^{p+1}$, we get
\[\langle|u|^{2q} u,\triangle u\rangle\leqslant -\int_{\mathbb{T}^n}|\nabla u|^2|u|^{2q}dx=-||\nabla U_q||^2_0,\]
and a similar relation holds for $q$ replaced by $p$. Therefore
\[
\begin{split}
\frac{d}{d\tau}H(u(\tau))&\leqslant -\frac{1}{2}||u||_2^2-\gamma_R||\nabla U_p||_0^2-\epsilon \gamma_I||\nabla U_q||_0^2-\epsilon \gamma_I\gamma_R|u|_{2p+2q+2}^{2p+2q+2}\\
&-\int_{\mathbb{T}^d}V(x)|\nabla u|^2dx+C_1||u||_0^2,
\end{split}
\]
where $C_1$ depends only on $|V|_{C^2}$. By this relation and  (\ref{norm0}), we have 
\begin{equation}
H(u(T))\leqslant H(u(0))+C_1TB_2^2,\quad\text{for any} \quad T>0.
\end{equation}
So
\begin{equation}
||u(T)||_1^{\backsim 2}\leqslant 2H(u(0))+2C_1TB_2^2,\quad \text{for any }\quad T>0.
\end{equation}
Simple calculation shows that
\[A_V^2u=(-\triangle)^2u-2V\triangle u-\nabla V\cdot\nabla u+(V^2-\triangle V)u.\]
 We consider 
\begin{equation}
\frac{d}{d\tau}\langle A_V^2 u,u\rangle=2\langle A_V^2 u, \triangle u-\gamma_R |u|^{2p}u-i\gamma_I|u|^{2q}u. \rangle\label{5.11}
\end{equation}
By the interpolation and Young inequality,  we have 
\begin{equation}
\begin{split}
\langle A_V^2u,\triangle u\rangle
&\leqslant -||u||_3^2+C_1(|V|)||u||_2^2+C_2(|V|_{C^1})||u||_1^2+C_3(|V|_{C^2})||u||_0^2\\
&\leqslant -||u||_3^2+C_1||u||_3^{\frac{4}{3}}||u||^2_0+C_2||u||_3^{\frac{2}{3}}||u||^2_0+C_3||u||_0^2\\
&\leqslant -\frac{3}{4}||u||_3^2+C(|V|_{C^2},||u||_0). 
\end{split}
\label{5.33}
\end{equation}
We deduce from integration by part and H\"older inequality that
\begin{equation}
-\langle (-\triangle)^2u, |u|^{2p}u\rangle\leqslant ||u||_3|(|u|^{2p}\nabla u)|_2\leqslant ||u||_3 |u|_{2pq_1}^{2q}|\nabla u|_{p_1},\label{5.13}
\end{equation}
where $p_1,q_1<\infty$ satisfy $1/p_1+1/q_1=1/2$.   Let $p_1$ and $q_1$ have the form
\[p_1=\frac{2d}{d-2s},\quad q_1=\frac{d}{s}.\]
We specify parameter $s$:
For $d\geqslant 3$,   choose $s=p(d-2)<\min\{d/2,2\}$;
for $d=1,2$, choose $s\in (0,\frac{1}{2})$.  Due to condition (\ref{cgl1c-2}), we have 
the Sobolev embeddings 
\[H^s(\mathbb{T}^d)\to L^{p_1}(\mathbb{T}^d) \quad\text{and} \quad H^1(\mathbb{T}^d)\to L^{2pq_1}(\mathbb{T}^d),\]
 implying that 
\[|\nabla u|_{p_1}\leqslant ||u||_{1+s},\quad |u|_{2pq_1}^{2p}\leqslant ||u||_1^{2p}.\]
Applying again the interpolation and Young inequality we find that for any $\delta>0$,

\begin{equation}
\begin{split}
-\langle \triangle^2 u , |u|^{2p}u \rangle&\leqslant ||u||_{3}||u||_{1+s}||u||_1^{2p}\\
&\leqslant C||u||_{3}^{1+\frac{1+s}{3}}||u||_0^{\frac{2-s}{3}}||u||_1^{2p}\\
&\leqslant \delta ||u||_{3}^2+C(\delta)(||u||_0^{\frac{2-s}{3}}||u||_1^{2p})^{\frac{2-s}{6}},
\end{split}
\end{equation}
We can deal with other terms in (\ref{5.11}) and (\ref{5.13}) similarly. With suitable choice of $\delta$, from the inequality above together with (\ref{5.33}), we can get that for any $T>0$
\begin{equation}
||u(T)||_2^{\backsim 2}+\int_0^T||u||_{3}^2d\tau\leqslant ||u(0)||_2^{\backsim2}+C(2,|V|_{C^4},T,B_2),
\end{equation}

By similar  argument, for any $m\geqslant 3$ and $T>0$ we can obtain
\[||u(T)||_m^{\backsim 2}+\int_0^T||u||_{m+1}^2d\tau\leqslant ||u(0)||_m^{\backsim2}+C(m,|V|_{C^{4m}},T,B_2),
\]
Then
\[||u(T)||_{m}\leqslant C(||u(0)||_m,|V|_{C^{4m}},m,T,B_2),\quad\text{for any} \;T>0.\]
This finishes the verification of assumption A.
\subsection{Verification of Assumption B} We follow \cite{Kuk11}. 
In equation (\ref{cgl1}) with $u\in H^2$, we pass to the $v$-variable, $v=\Psi(u)\in h^2$:
\begin{equation}
\dot{v}_k+i\epsilon^{-1}\lambda_k=P_k(v),\quad k\geqslant 1.
\end{equation}
Here
\[P_k=P_k^1+P_k^2+P_k^3,\]
where $P^1$, $P^2$ and $P^3$ are, correspondingly,  the linear, nonlinear dissipative and nonlinear  hamiltonian parts of the perturbation:
\[P^1(v)=\Psi(\triangle u),\quad P^2(v)=-\gamma_R\Psi(|u|^{2p}u),\quad P^3(v)=-i\gamma_I\Psi(|u|^{2q}u),\]
with $u=\Psi^{-1}(v)$. Following the procedure in Section 3, the effective equations for (\ref{cgl1}) has the form:
\begin{equation}
\dot{v}=\sum_{i=1}^3R^i(v),
\end{equation}
where 
\[ R^i(v)=\int_{\mathbb{T}^{\infty}}\Phi_{-\theta}P^i(\Phi_{\theta})\Big)d\theta,\quad i=1, 2, 3.\]
Consider the operator
\[\mathcal{L}:=\Psi\circ(-\triangle)\circ\Psi^{-1}=\Psi\circ(A_V-V)\circ\Psi^{-1}:=\hat{A}-\Psi\circ V\circ\Psi^{-1}:=
\hat{A}-\mathcal{L}^0.\]
Clearly, $\hat{A}$ is the diagonal operator $\hat{A}=\text{diag}\{\lambda_j\left(\begin{array}{cc}1 & 0 \\0 & 1\end{array}\right),\;j\geqslant 1\}$.  By Lemma~1.1, $\mathcal{L}^0=\Psi\circ V\circ \Psi^{-1} $ defines bounded maps
\[\mathcal{L}^0:\quad h^m\to h^m,\quad \forall m\leqslant n,\]
and in the space $h^0$ the operator $\mathcal{L}^0$ is self-adjoint.
Since $\hat{A}$ commutes with the rotation $\Phi_{\theta}$, then 
\begin{equation}
\begin{split}
R^1&=-\int_{\mathbb{T}^{\infty}}\Phi_{-\theta} \hat{A}\Phi_{\theta}vd\theta+\int_{\mathbb{T}^{\infty}}\Phi_{-\theta}\mathcal{L}^0(\Phi_{\theta}v)d\theta\\
&=-\hat{A}v+R^0(v),\quad R^0(v)=\int_{\mathbb{T}^{\infty}}\Phi_{-\theta}\mathcal{L}^0(\Phi_{\theta}v)d\theta.
\end{split}
\end{equation}
Since for $v=(v_1,v_2,\dots)$, we have
\[\mathcal{L}^0(v)_j=\sum_{i=0}^{+\infty}\langle V(x)v_i\varphi_i(x),\varphi_j(x)\rangle,\quad j\geqslant1,\]
then,
\[
R^0_k(v)=\sum_{j=1}^{+\infty}\int_{\mathbb{T}^{\infty}} \langle  V(x)v_je^{i\theta_j}\varphi_j(x), e^{i\theta_k} \varphi_k(x)\rangle d\theta=v_k\langle V\varphi_k, \varphi_k\rangle.
\]
That is, 
\begin{equation}
R^1=\text{diag}\;\{-\lambda_k+M_k, k\geqslant 1\},\quad M_k=\langle V\varphi_k,\varphi_k\rangle.
\end{equation}
The term $R^2(v)$ is defined as an integral with the integrand
\[\Phi_{-\theta}P^2\Phi_{\theta}(v)=-\gamma_{R}\Phi_{-\theta}\Psi(f_p(|u|^2)u)|_{u=\Psi^{-1}\Phi_{\theta}v}:=F_{\theta}(v).
\]
Define $\mathcal{H}(u)=\int \mathcal{F}(|u|^2)dx$, where $\mathcal{F}^{\prime}=\frac{1}{2}f_p$. Then $\nabla \mathcal{H}(u)=f_p(|u|^2)u$. Denoting $\Psi^{-1}\Phi_{\theta}=L_{\theta}$, we have
\[F_{\theta}(v)=-\gamma_{R}L^{*}_{\theta}\nabla \mathcal{H}(u)|_{u=L_{\theta}(v)}=-\gamma_{R}\nabla(\mathcal{H}\circ L_{\theta}(v)).
\]
So
 \[R^2(v)=-\gamma_{R}\nabla_{v}\Big(\int_{\mathbb{T}^{\infty}}(\mathcal{H}\circ \Psi^{-1})(\Psi_{\theta}v)d\theta\Big)=-\gamma_{R}\nabla_v\langle \mathcal{H}\circ\Psi^{-1}\rangle.
 \]
Similarly, we have $R^3=-i\gamma_I\nabla_v\langle \mathcal{G}\circ\Psi^{-1}\rangle$ with $\nabla \mathcal{G}(u)=f_q(|u|^2)u$.  Since $\langle \mathcal{G}\circ\Psi^{-1}\rangle$ is a function only of the action $(I_1,\dots)$, we have that $\nabla_{v_k}\langle \mathcal{G}\circ \Psi^{-1}\rangle$ is proportional to $v_k$. Then $v_k\cdot R^3_k(v)=0$.  That is, it contributes a zero term in the averaged equation. 
Hence we could set  the effective equation to be\[\dot{v}=R^1(v)+R^2(v).\]
It is a quasi-linear heat equation,  written in Fourier coefficients, which is known to be locally well posed. This verifies assumption B.
\section*{Acknowledgments}
Firstly, the author want to thank his PhD supervisor Sergei Kuksin for formulation of the problem and guidance. He is also grateful to professor Dario Bambusi for useful suggestions and pointing out a flaw in the original manuscript. Finally, he  would like to thank all of the staff and faculty at   CMLS of \'Ecole Polytechnique for their support. 

\section*{Appendix}
Consider the $l_2$-space of sequences $x=(x_1,x_2,\dots)$. 
The following lemma is a slight modification of the well known  theorem of Whitney  \cite{HW42}.

\noindent {\bf Lemma A.} For any $n\in\mathbb{N}$,  let $f\in C^{\infty}(l_2)$ be  even in $n$ variables, i.e.
\[f(x_1,\dots,x_i,\dots)=f(x_1,\dots,-x_i,\dots),\quad i=1,2,\dots, n.\]
Then   there exists $g_n\in C^{\infty}(l_2)$ such that \[g_n(x_1^2,\dots,x_n^2,x_{n+1},\dots)=f(x_1,x_2,\dots).\]

\begin{proof}\quad For $n=1$,  we define $g_1(x_1,x_2,\dots)=f(x_1^{\frac{1}{2}},x_2,\dots)$.  Since $f$ is even with respect to $x_1$, for any $s\in\mathbb{N}$, we have 
\[f(x_1,x_2,\dots)=f(\hat{x})+f_1(\hat{x})x_1^2+\cdots+f_{s-1}(\hat{x})x_1^{2s-2}+\phi(x)x_1^{2s},\]
where $\hat{x}=(0,x_2,\dots)$, $f_i=[(2i)!]^{-1}\partial_{x_1}^{2i}f(\hat{x})$ and $\phi(x)$ is smooth when $x_1\neq0$, even with respect to $x_1$, and satisfies
\[\lim_{x_1\to 0} x_1^k\partial_{x_1}^k\phi(x)=0,\quad k=1,\dots,2s.\eqno{(A.1)}\]
Set $\psi(x)=\phi(x_1^{\frac{1}{2}},x_2,\dots)$, then 
\[g_1(x)=f(\hat{x})+f_1(\hat{x})x_1+\cdots+f_{s-1}(\hat{x})x_1^{s-1}+\psi(x)x_1^s.\]
 We wish  to check that  $g_1(x)$ is $C^s$-smooth with respect to $x_1$. It is sufficient to prove that the limits 
$\lim_{x_1\to0}x_1^k\partial_{x_1}^k\psi(x)$,  $k=1,\dots,s$, exist and are finite. Differentiating  $\psi(x_1^2,x_2,\dots)=\phi(x)$ with respect to $x_1$, we get that  there are some constants $a_{ki}$ such that 
\[\partial_{x_1}^k\phi(x)=2^k x_1^k\partial_{x_1}^k\psi(x_1^2,x_2,\dots)+\sum_{1\leqslant i\leqslant k/2}a_{ki}x_1^{k-2i}\partial_{x_1}^{k-i}\psi(x_1^2,x_2,\dots),\quad k=1,\dots, s.\]
Solving these equation successively for $x^{2k}_1\partial_{x_1}^k\psi$,  $k=1,\dots,s$, we obtain that  there are some constant $\beta_{ki}$ such that 
\[x_1^{2k}\partial_{x_1}^k\psi(x_1^2,x_2,\dots)=\sum_{0\leqslant i\leqslant k}\beta_{ki}x_1^{k-i}\partial_{x_1}^{k-i}\phi(x).\]
 
By ($A.1$), we know the $\lim_{x_1\to0}x_1^k\partial_{x_1}^k\psi(x)$, $k=1,\dots,s$, exist and are finite. So $g_1(x)$ is $C^s$ -smooth. Since $s$ is arbitrary and $g_1(x)$ defined in a unique way, we have $g_1\in C^{\infty}(l^2)$ and $g_1(x_1^2,x_2,\dots)=f(x_1,x_2,\dots)$. This prove the statement of the lemma for $n=1$. 

 For $n\geqslant 2$,  the assertion of the lemma can be prove by  induction.  Assume we have proved the lemma for $m=n-1$. Then there exists $g_{n-1}\in C^{\infty}(l_2)$ such that
$g_{n-1}(x_1^2,\dots,x_{n-1}^2,x_n)=f(x_1,x_2,\dots)$ and $g_{n-1}$ is even in variable $x_n$. Applying what we have proved for $m=1$ to  $g_{n-1}$ with respect to $x_n$, we  get the assertion for $m=n$.
\end{proof}

 \bibliography{CGL.bib}

\end{document}